\documentclass{article}
\usepackage[a4paper]{geometry}
\usepackage{amsfonts,amsmath,enumitem,graphicx,amssymb,amsthm,tikz}
\usepackage[initials]{amsrefs}
\usepackage{url}

\newtheorem{theorem}{Theorem}[section]
\newtheorem{lemma}[theorem]{Lemma}

\theoremstyle{definition}

\numberwithin{equation}{section}
\numberwithin{figure}{section}

\renewcommand{\geq}{\geqslant}
\renewcommand{\leq}{\leqslant}

\makeatletter
\renewcommand\section{\@startsection {section}{1}{\z@}%
                                   {-3.5ex \@plus -1ex \@minus -.2ex}%
                                   {1.3ex \@plus.2ex}%
                                   {\normalfont\large\scshape}}
\makeatother

\title{ \vspace{-5ex}\bf \Large Repeated compositions of M\"obius transformations}

\author{Matthew Jacques and Ian Short\footnote{Department of Mathematics and Statistics, The Open University, Milton Keynes, MK7 6AA, United Kingdom 28~April~2015}}

\date{\vspace{-5ex}}

\begin{document}

\maketitle

\begin{abstract}
This paper considers a class of dynamical systems generated by finite sets of M\"obius transformations acting on the unit disc. Compositions of such M\"obius transformations give rise to sequences of transformations that are used in the theory of continued fractions. In that theory, the distinction between sequences of limit-point type and sequences of limit-disc type is of central importance. We prove that sequences of limit-disc type only arise in particular circumstances, and we give necessary and sufficient conditions for a sequence to be of limit-disc type. We also calculate the Hausdorff dimension of the set of sequences of limit-disc type in some significant cases. Finally, we obtain strong and complete results on the convergence of these dynamical systems.
\end{abstract}

\section{Introduction}

The object of this paper is to give a precise description of a class of nonautonomous dynamical systems generated by M\"obius transformations. The motivation for our study comes from discrete dynamical systems, where one is concerned with the iterates of a map $f:X\to X$. In our more general context, given any collection $\mathcal{F}$ of self maps of $X$, we define a  \emph{composition sequence} generated by $\mathcal{F}$ to be a sequence $(F_n)$, where $F_n=f_1f_2\dotsb f_n$ and $f_i\in \mathcal{F}$. Composition sequences generated by sets of analytic self maps of complex domains have received much attention; see, for example, \cites{BaRi1990,BeCaMiNg2004,KeLa2006,Lo1990}. There has been particular focus on generating sets composed of M\"obius transformations that map a disc within itself -- see \cites{Ae1990,Be2001,Be2003,JaSh2016,Lo2007,LoWa2008} -- partly because of applications to the theory of continued fractions.  Here we give a detailed study of composition sequences generated by  \emph{finite} sets of M\"obius transformations, which goes far beyond the existing literature in completeness and precision. Furthermore, the geometric approach we take generalises easily to higher dimensions (for simplicity, however, we  present all our results in two dimensions).

 We understand a M\"obius transformation to be a map of the form $z\mapsto (az+b)/(cz+d)$, where $a$, $b$, $c$ and $d$ are complex numbers that satisfy $ad-bc\neq 0$. These transformations act on the extended complex plane $\overline{\mathbb{C}}=\mathbb{C}\cup\{\infty\}$ (with the usual conventions about the point $\infty$). Our focus is on composition sequences generated by subsets of the class $\mathcal{M}(\mathbb{D})$ of those M\"obius transformations that map the unit disc $\mathbb{D}$ strictly inside itself. By conjugation, the results that follow  remain valid if we replace $\mathbb{D}$ with any other disc (by \emph{disc} we mean a disc of positive radius in $\overline{\mathbb{C}}$ using the chordal metric, other than $\overline{\mathbb{C}}$ itself). Let us say that a composition sequence is \emph{finitely generated} if the corresponding generating set is finite.
 
If $(F_n)$ is a composition sequence generated by a subset of $\mathcal{M}(\mathbb{D})$, then clearly the discs
\[
\mathbb{D} \supset F_1(\mathbb{D}) \supset F_2(\mathbb{D}) \supset \dotsb
\]
are nested, and the intersection $\bigcap F_n(\overline{\mathbb{D}})$ is either a single point or a closed disc. In the first case we say that the composition sequence is of \emph{limit-point type}, and in the second case we say that it is of \emph{limit-disc type}. The distinction between limit-point type and limit-disc type is of central importance in continued fraction theory because composition sequences of limit-point type have strong convergence properties. We will see that for composition sequences generated by finite subsets of $\mathcal{M}(\mathbb{D})$, limit-disc type really is quite special, and we supply necessary and sufficient conditions for limit-disc type to occur. 

To state our first result we introduce a new concept: a composition sequence $(F_n)$ generated by a subset of $\mathcal{M}(\mathbb{D})$ is said to be of \emph{limit-tangent type} if  all but a finite number of discs from the sequence~$\mathbb{D} \supset F_1(\mathbb{D}) \supset F_2(\mathbb{D}) \supset\dotsb$ share a single common boundary point. 

\begin{theorem}\label{dgi}
Let $\mathcal{F}$ be a finite subcollection of $\mathcal{M}(\mathbb{D})$. Any composition sequence generated by~$\mathcal{F}$ of limit-disc type is of limit-tangent type.
\end{theorem}

That a composition sequence $(F_n)$ is of limit-tangent type implies that there is a point $p$ and a sequence $(z_n)$ in $\partial \mathbb{D}$ such that $p=F_n(z_n)$ for sufficiently large values of $n$. It follows that $f_n(z_n)=z_{n-1}$ for large $n$. We know that $f_n(\mathbb{D})\subset \mathbb{D}$ and $f_n(\mathbb{D})\neq \mathbb{D}$, so we deduce that $f_n(\mathbb{D})$ is internally tangent to $\mathbb{D}$ at the point $z_{n-1}$. Let us now choose any element $f$ of $\mathcal{M}(\mathbb{D})$; the disc $f(\mathbb{D})$ is either internally tangent to $\mathbb{D}$ at a unique point, or else the boundaries of $f(\mathbb{D})$ and $\mathbb{D}$ do not meet. In the former case, we define $\alpha_f$ and $\beta_f$ to be the unique points in $\partial \mathbb{D}$ such that $f(\alpha_f)=\beta_f$, and in the latter case, we define $\alpha_f=0$ and $\beta_f=\infty$ (for reasons of convenience to emerge shortly). The preceding theorem tells us that in order for the  sequence $(F_n)$ to be of limit-disc type, we need $\alpha_{f_{n-1}}=\beta_{f_n}$ for sufficiently large values of $n$. How likely one is to find such a composition sequence is best illustrated by means of a directed graph $T(\mathcal{F})$, which we call the \emph{tangency graph} of $\mathcal{F}$, and which is defined as follows. The vertices of $T(\mathcal{F})$ are the elements of $\mathcal{F}$. There is a directed edge from vertex $f$ to vertex $g$ if $\alpha_f=\beta_g$. Clearly, any vertex $f$ for which $f(\mathbb{D})$ and $\mathbb{D}$ are not internally tangent is an isolated vertex.

Let us consider an example. One can easily check that a M\"obius transformation $f(z)=a/(b+z)$, where $a\neq 0$, has the property that $f(\mathbb{D})$ is contained within and internally tangent to~$\mathbb{D}$ if and only if $|b|=1+|a|$. Let 
\[
g(z)=\frac{\tfrac12}{\tfrac32+z},\quad  h(z)=\frac{\tfrac12}{-\tfrac32+z},\quad\text{and}\quad k(z)=\frac{-\tfrac12}{-\tfrac32+z}.
\]
Each of these maps satisfies the condition $|b|=1+|a|$. Observe that $g(-1)=1$, $h(1)=-1$ and $k(1)=1$, which implies that $\alpha_g=-1$, $\beta_g=1$, $\alpha_h=1$, and so forth. The tangency graph of $\{g,h,k\}$ is shown in Figure~\ref{poa}.

\pgfmathsetmacro{\L}{2}
\pgfmathsetmacro{\yL}{{\L*sqrt(3)}}
\pgfmathsetmacro{\sgap}{0.3cm}
\pgfmathsetmacro{\gap}{0.4}
\pgfmathsetmacro{\xgap}{0.5*\gap}
\pgfmathsetmacro{\ygap}{{\xgap*sqrt(3)}}

\begin{figure}[ht]
\centering
\includegraphics[scale=0.9]{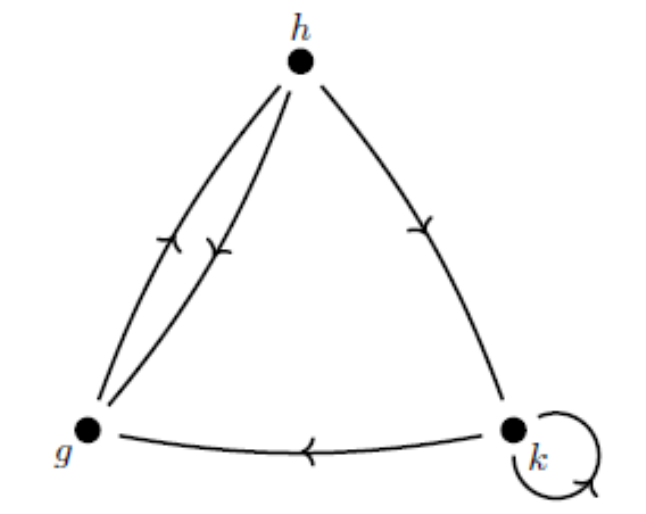}
\caption{Tangency graph of $\{g,h,k\}$}
\label{poa}
\end{figure}

From the tangency graph we can see that a composition sequence $F_n=f_1\dotsb f_n$ generated by~$\{g,h,k\}$ is \emph{not} of limit-tangent type if and only if $(f_n,f_{n+1})$ is equal to either~$(g,g)$,~$(h,h)$,~$(g,k)$, or~$(k,h)$ for infinitely many positive integers $n$. 

We have shown that in order for a composition sequence $(F_n)$ to be of limit-disc type, the sequence of vertices $(f_n)$ in the tangency graph must eventually form an infinite path. The converse fails though: not all infinite paths in the tangency graph correspond to sequences of limit-disc type. To determine which paths arise from composition sequences of limit-disc type, we need to look at the derivatives of the maps $f_n$ at tangency points. Let $\gamma_f = 1/|f'(\alpha_f)|$, where $f \in\mathcal{M}(\mathbb{D})$. The next theorem gives necessary and sufficient conditions for a composition sequence to be of limit-disc type.

\begin{theorem}\label{acu}
Let $\mathcal{F}$ be a finite subcollection of $\mathcal{M}(\mathbb{D})$. A composition sequence $F_n=f_1\dotsb f_n$ generated by  $\mathcal{F}$ is of limit-disc type if and only if 
\begin{enumerate}[label=\emph{(\roman*)},leftmargin=20pt,topsep=0pt,itemsep=0pt]
\item $\alpha_{f_{n-1}}=\beta_{f_n}$ for all but finitely many positive integers $n$, and
\item $\displaystyle\sum_{n=1}^\infty \gamma_{f_1}\dotsb \gamma_{f_n} < +\infty$.
\end{enumerate}
\end{theorem}

For the example illustrated in Figure~\ref{poa}, one can check that $\gamma_g=\gamma_h=\gamma_k=\tfrac12$. It follows that any composition sequence generated by $\{g,h,k\}$ satisfies condition~(ii) of Theorem~\ref{acu}, so such a sequence is of limit-disc type if and only if it is of limit-tangent type.

Next we turn to the following question: given a finite subcollection $\mathcal{F}$ of $\mathcal{M}(\mathbb{D})$, how many of the composition sequences generated by $\mathcal{F}$ are of limit-disc type? Using the tangency graph we can formalise this question, and answer it, at least in part, using known techniques. More precisely, we will define a metric on the set $\Omega(\mathcal{F})$ of all infinite sequences $(f_n)$ with $f_n\in\mathcal{F}$, and evaluate the Hausdorff dimension of the subset $\Lambda(\mathcal{F})$  composed of sequences $(f_n)$ for which the corresponding composition sequences $F_n=f_1\dotsb f_n$ are of limit-disc type. The value of the Hausdorff dimension will depend on the metric we choose; a standard choice for a metric on the set of sequences from a finite alphabet (which we adopt) is as follows. Given sequences $(f_n)$ and $(g_n)$ in $\mathcal{F}$, we define 
\[
d\big((f_n),(g_n)\big)= \frac{1}{|\mathcal{F}|^k},\quad\text{where}\quad k= \min\{ i\in\mathbb{N}\,:\, f_i\neq g_i\}.
\]
Then $(\Omega(\mathcal{F}),d)$ is a compact metric space. 

A \emph{cycle} in a directed graph is a sequence $x_0,\dots,x_n$ of vertices in the graph such that $x_0=x_n$ and there is a directed edge from $x_{i-1}$ to $x_i$, for $i=1,\dots,n$.  Let $\rho(\mathcal{F})$ denote the spectral radius of the adjacency matrix of $T(\mathcal{F})$. We can now state our first theorem on Hausdorff dimension.

\begin{theorem}\label{can}
Let $\mathcal{F}$ be a finite subcollection of $\mathcal{M}(\mathbb{D})$, and suppose that $\gamma_f<1$ for  each $f\in\mathcal{F}$. If~$\mathcal{F}$ contains a cycle, then 
\[
\textnormal{dim}\,\Lambda(\mathcal{F})=\dfrac{\log \rho(\mathcal{F})}{\log |\mathcal{F}|},
\]
and otherwise $\textnormal{dim}\,\Lambda(\mathcal{F})=0$.
\end{theorem}

Returning again to the example illustrated in Figure~\ref{poa}, one can check that the spectral radius of the adjacency matrix of the tangency graph of $\mathcal{G}=\{g,h,k\}$ is $\tfrac12(1+\sqrt5)$. Also, $\gamma_g,\gamma_h,\gamma_k<1$. Therefore $\text{dim}\,\Lambda(\mathcal{G})=(\log(1+\sqrt5)-\log 2)/\log 3$.

As we will see, Theorem~\ref{can} is a corollary of Theorem~\ref{acu} and a well-known result on Hausdorff dimensions of paths in directed graphs. Also, because it assumes that condition (ii) from Theorem~\ref{acu} is always satisfied, it is effectively about the Hausdorff dimension of the set of composition sequences of limit-\emph{tangent} type. In contrast, the next theorem determines $\textnormal{dim}\,\Lambda(\mathcal{F})$ under the assumption that condition (i) from Theorem~\ref{acu} is always satisfied. Condition (i) is always satisfied precisely when $T(\mathcal{F})$ is a complete directed graph, that is, when there is a directed edge between every pair of (not necessarily distinct) vertices of $T(\mathcal{F})$. (Notice that, according to this definition, but in contrast to some definitions elsewhere, a complete directed graph includes a directed edge from each vertex to itself.) It is straightforward to show that $T(\mathcal{F})$ is a complete directed graph if and only if all elements of $\mathcal{F}$ share a common fixed point on $\partial\mathbb{D}$.

\begin{theorem}\label{cun}
Let $\mathcal{F}$ be a nonempty finite subcollection of $\mathcal{M}(\mathbb{D})$ for which $T(\mathcal{F})$ is the complete directed graph on $|\mathcal{F}|$ vertices. Then $\textnormal{dim}\,\Lambda(\mathcal{F})=0$ if $\gamma_f\geq1$ for all $f\in\mathcal{F}$; $\textnormal{dim}\,\Lambda(\mathcal{F})=1$ if $\gamma_f<1$ for all $f\in\mathcal{F}$; and otherwise
\[
\textnormal{dim}\,\Lambda(\mathcal{F})=\min_{s\geq 0}\frac{\log \left(\sum_{f\in\mathcal{F}}\gamma_f^{-s}\right)}{\log|\mathcal{F}|}. 
\]
\end{theorem}

Given an arbitrary finite subset  $\mathcal{F}$ of $\mathcal{M}(\mathbb{D})$, Theorems~\ref{can} and~\ref{cun} provide upper bounds for  the Hausdorff dimension of $\Lambda(\mathcal{F})$. The problem of determining $\textnormal{dim}\,\Lambda(\mathcal{F})$ in general remains open.

So far we have focused on determining those composition sequences of limit-disc type, and we have ignored other questions of convergence. In fact, it is known already that any composition sequence generated by a finite subset of $\mathcal{M}(\mathbb{D})$ converges locally uniformly in $\mathbb{D}$ to a constant (this follows, for example, from \cite{Lo2007}*{Theorems~3.8 and~3.10}). Here we will prove some stronger theorems on convergence for finitely-generated composition sequences. 

Before we can state our results on convergence, we must recall some three-dimensional hyperbolic geometry. Refer to a text such as  \cite{Be1995} for a detailed introduction to hyperbolic geometry. Let $\mathbb{H}^3=\{(x,y,t)\in\mathbb{R}^3\,:\,t>0\}$. When equipped with the Riemann density $ds/t$, this set becomes a model of three-dimensional hyperbolic space; we denote the hyperbolic metric by $\rho$. We identify the complex plane $\mathbb{C}$ with the plane $t=0$ in $\mathbb{R}^3$ in the obvious manner, so that $\overline{\mathbb{C}}$ can be thought of as the ideal boundary of $(\mathbb{H}^3,\rho)$. The compact set $\overline{\mathbb{H}^3}=\mathbb{H}^3\cup\overline{\mathbb{C}}$ can be endowed with the chordal metric. We extend the action of the M\"obius group from $\overline{\mathbb{C}}$ to $\mathbb{H}^3$ in the usual way, so that each M\"obius transformation is an orientation-preserving hyperbolic isometry of $(\mathbb{H}^3,\rho)$. Let $j$ denote the distinguished point $(0,0,1)$ in $\mathbb{H}^3$. Notice that here, and throughout the paper, we use the bar notation $\overline{X}$ to represent the closure with respect to the chordal metric of a subset $X$ of $\overline{\mathbb{H}^3}$. 

A sequence of M\"obius transformations $(F_n)$ is called an \emph{escaping sequence} if $\rho(j,F_n(j))\to \infty$ as $n\to\infty$ (the choice of the particular point $j$ from $\mathbb{H}^3$ is not important here, or in any of the definitions to follow). The term `escaping sequence' is taken from the theory of transcendental dynamics, where it has a somewhat similar meaning. Escaping sequences also play an important role in the theory of Kleinian groups; after all, any sequence of distinct elements from a Kleinian group is an escaping sequence. Escaping sequences feature significantly in continued fraction theory too; however, in that context they are usually called \emph{restrained sequences}, and are often defined differently (in terms of the action of the sequence on $\overline{\mathbb{C}}$ rather than on $\mathbb{H}^3$; see \cite{LoWa2008}*{Definition~2.6}). We say that an escaping sequence $(F_n)$ is a \emph{rapid-escape sequence} if $\sum \exp(-\rho(j,F_n(j))) <+\infty$. This type of sum is familiar to Kleinian group theorists: it relates to the critical exponent of a Kleinian group, a connection that has been explored before in \cite{Sh2012}. 

We have one final definition: an escaping sequence $(F_n)$ is said to \emph{converge ideally} to a point $p$ in $\overline{\mathbb{C}}$ if the sequence $(F_n(j))$ converges to $p$ in the chordal metric. In the continued fractions literature, sequences that converge ideally are usually said to \emph{converge generally}; we have changed the terminology to fit our context.

Our first theorem on convergence demonstrates that composition sequences generated by finite subsets of $\mathcal{M}(\mathbb{D})$ are rapid-escape sequences. In fact, the theorem is more general than this, as it does not assume that the generating set is finite. Let $\text{rad}(D)$ denote the Euclidean radius of a Euclidean disc $D$.

\begin{theorem}\label{ibi}
Let $(f_n)$ be a sequence of M\"obius transformations such that~$f_n(\mathbb{D})\subset \mathbb{D}$ and $\textnormal{rad}(f_n(\mathbb{D}))<\delta$, for $n=1,2,\dotsc$, where  $0<\delta<1$. Then $F_n=f_1\dotsb f_n$ is a rapid-escape sequence.
\end{theorem}

The property of being a rapid-escape sequence has strong implications for convergence. By placing further, relatively mild conditions on a rapid-escape sequence (such as assuming that the sequence is a finitely-generated composition sequence) one can prove that the sequence converges ideally. What is more, if $(F_n)$ is a rapid-escape sequence that converges ideally to a point $p$, then one can show that the set of points $z$ in $\overline{\mathbb{C}}$ such that $F_n(z)\nrightarrow p$ as $n\to\infty$ has Hausdorff dimension at most~1 (see \cite{Sh2012}*{Corollary~3.5}). For finitely-generated composition sequences of the type that interest us, we have the following precise result.

\begin{theorem}\label{mwd}
Let $F_n=f_1\dotsb f_n$ be a composition sequence generated by a finite subcollection $\mathcal{F}$ of $\mathcal{M}(\mathbb{D})$. Then $(F_n)$ converges ideally to a point $q$ in $\overline{\mathbb{D}}$. Furthermore,  $(F_n)$ converges locally uniformly to $q$ on $\overline{\mathbb{D}}\setminus X$, where 
\[
X=\{\alpha_f\in\partial\mathbb{D}\,:\, \text{$f_n=f$ for infinitely many integers $n$}\}.
\]
For each  $x\in X$, the sequence $(F_n(x))$ converges to $q$ if and only if $(F_n)$ is of limit-point type.
\end{theorem}

In fact, it is relatively straightforward to prove that if $(F_n)$ is of limit-disc type, and $|X|>1$, then the sequence $(F_n(x))$ diverges.

Theorem~\ref{mwd} gives us a complete understanding of the dynamical system $(F_n)$. The part about converging ideally is known from other more general results (again, we refer the reader to \cite[Theorems~3.8 and~3.10]{Lo2007}).

Throughout the paper we adopt the convention that $A \subset B$ and $B\supset A$ both mean that $A$ is a subset of $B$, possibly equal to $B$.

\section{Sequences of limit-tangent type}

Let us say that a nested sequence of open discs $D_1\supset D_2\supset \dotsb$ is \emph{eventually tangent} if all but finitely many of the discs $D_n$ share a common boundary point. To prove Theorem~\ref{dgi}, we need the following elementary lemma, the proof of which, towards the end, uses the even-more elementary observation that if two distinct open discs $D$ and $E$ satisfy $D\subset E$, then they can have at most one point of tangency.

\begin{lemma}\label{emr}
Let $D_1\supset D_2\supset \dotsb$ be a nested sequence of open discs in the plane, no two of which are equal. The sequence is eventually tangent if and only if $D_{n}$ is tangent to $D_{n+2}$ for all but finitely many values of $n$.
\end{lemma}
\begin{proof}
Suppose that the sequence $(D_n)$ is eventually tangent. Then there is a point $p$ in the plane and a positive integer $m$ such that $p\in \partial D_n$ when $n\geq m$. Therefore $D_n$ and $D_{n+2}$ share a common boundary point, provided $n\geq m$, and because the two discs are distinct and satisfy $D_{n+2}\subset D_n$, they must be tangent to one another.

Conversely, suppose that there is a positive integer $m$ such that $D_{n}$ is tangent to $D_{n+2}$ when $n\geq m$. For any particular integer $k$, with $k\geq m$, let $p$ be the point of tangency of $D_k$ and $D_{k+2}$, and let $c$ and $r$ be the Euclidean centre and radius of $D_{k+1}$, respectively. Now observe that, first, $p\in \partial D_k$ and $D_{k}\supset D_{k+1}$, so $|p-c|\geq r$, and second, $p\in \partial D_{k+2}$ and $D_{k+2}\subset D_{k+1}$, so $|p-c|\leq r$. Therefore $|p-c|=r$, which implies that $D_k$, $D_{k+1}$ and $D_{k+2}$ have a (unique) common point of tangency. It follows that all the discs $D_n$, for $n\geq m$, have a common point of tangency.
\end{proof}

We remark that the lemma remains true if  we replace `$D_{n}$ is tangent to $D_{n+2}$' with `$D_{n}$ is tangent to $D_{n+q}$', for any positive integer $q>1$.

We also use the following theorem, which is a simplified version of \cite[Theorem~1]{Lo1990}. (Significantly more general versions of this result are known; see, for example, \cite{BeCaMiNg2004}).

\begin{theorem}\label{bia}
Suppose that $K$ is a compact subset of $\mathbb{D}$ and $\mathcal{F}$ is a collection of analytic maps from $\mathbb{D}$ to $K$. Then any composition sequence generated by $\mathcal{F}$ converges locally uniformly in $\mathbb{D}$ to a constant.\qed
\end{theorem}

Let us now prove Theorem~\ref{dgi}.

\begin{proof}[Proof of Theorem~\ref{dgi}]
We will prove the contrapositive to the assertion made in the theorem; that is, we will prove that if $(F_n)$ is \emph{not} of limit-tangent type, then it is of limit-point type.

Suppose then that $(F_n)$ is not of limit-tangent type. By Lemma~\ref{emr}, there is an infinite sequence of positive integers $n_1<n_2<\dotsb$, where $n_{i+1}>n_i+1$ for all $i$, such that $F_{n_i+2}(\mathbb{D})$ is not tangent to $F_{n_i}(\mathbb{D})$ for each integer $n_i$. Therefore $f_{n_i+1}f_{n_i+2}(\mathbb{D})$ is not tangent to $\mathbb{D}$ for each integer $n_i$. Let $K$ be the union of all the sets $fg(\overline{\mathbb{D}})$, where $f,g\in \mathcal{F}$, such that $fg(\mathbb{D})$ is not tangent to $\mathbb{D}$. Since $\mathcal{F}$ is finite, $K$ is a compact subset of $\mathbb{D}$.

Let $g_1=f_1\dotsb f_{n_1}$ and $g_{i+1}=f_{n_i+1}\dotsb f_{n_{i+1}}$ for each positive integer $i$. Then
\[
g_{i+1}(\mathbb{D})=f_{n_i+1}\dotsb f_{n_{i+1}}(\mathbb{D})\subset f_{n_i+1}f_{n_i+2}(\mathbb{D})\subset K.
\]
Let $G_n=g_1\dotsb g_n$. As $g_n(\overline{\mathbb{D}})\subset K$, we can apply Theorem~\ref{bia} to see that $(G_n)$ converges uniformly on $\overline{\mathbb{D}}$ to a constant. That is, $(G_n)$ is of limit-point type. Since $(G_n)$ is a subsequence of $(F_n)$, we conclude that  $(F_n)$ is also of limit-point type.
\end{proof}

\section{Necessary and sufficient conditions for limit-disc type}

In this section we prove Theorem~\ref{acu}.

Let $\mathbb{K}$ denote the right half-plane. Suppose that $h$ is a M\"obius transformation that satisfies $h(\mathbb{K})\subset \mathbb{K}$ and $h(\infty)=\infty$. Then $h(z)=az+b$, where $a>0$ and $\text{Re}(b)\geq 0$, and $h(\mathbb{K})=\mathbb{K}$ if and only if $\text{Re}(b)=0$.

The next lemma is about composition sequences generated by the collection of those M\"obius transformations that map $\mathbb{K}$ strictly inside itself. We say that a composition sequence $(H_n)$ of this sort is of \emph{limit-disc type} if $\bigcap H_n(\overline{\mathbb{K}})$ is a disc.

\begin{lemma}\label{pbe}
Consider the finitely-generated composition sequence $H_n=h_1\dotsb h_n$, where $h_n(z)=a_nz+b_n$, with $a_n>0$  and $\textnormal{Re}(b_n)>0$. Then $(H_n)$ is of limit-disc type if and only if 
\[
\sum_{n=1}^\infty a_1\dotsb a_{n} < +\infty.
 \] 
 Furthermore, if $(H_n)$ is of limit-disc type and converges ideally to a point $q$, then $q\neq \infty$.
\end{lemma}
\begin{proof}
Observe that $H_n(z) = a_1\dotsb a_nz + \sum_{k=1}^n a_1\dotsb a_{k-1}b_k$ (where $a_1\dotsb a_0=1$). Therefore
\[
H_n(\mathbb{K}) = t_n + \mathbb{K},\quad\text{where}\quad t_n=\sum_{k=1}^n a_1\dotsb a_{k-1}\text{Re}(b_k).
\]
Hence $(H_n)$ is of limit-disc type if and only if $(t_n)$ converges to a finite number. As $(H_n)$ is finitely generated, we see that the positive numbers $\text{Re}(b_k)$ take only finitely many values. The first assertion of the theorem follows.

 For the second assertion, we know that $\sum_n a_1\dotsb a_n <+\infty$, so $a_1\dotsb a_n\to 0$ as $n\to\infty$. Therefore $H_n(j)\to q$, where $q=\sum_k a_1\dotsb a_{k-1}b_k$, as $n\to\infty$. In particular, $q\neq \infty$.
\end{proof}

The chordal derivative $f^\#$ of a M\"obius transformation $f$ is given by
\[
f^\#(z) = \frac{1+|z|^2}{1+|f(z)|^2}|f'(z)|,
\]
with the usual conventions about $\infty$ values. Notice in particular that if $|z|=1$ and $|f(z)|=1$, then $f^\#(z)=|f'(z)|$.

\begin{lemma}\label{abc}
Let $F_n=f_1\dotsb f_n$ be a composition sequence generated by a finite subset of $\mathcal{M}(\mathbb{D})$ such that $f_n(z_n)=z_{n-1}$ for each $n$, where $z_0,z_1,\dotsc$ are points on the unit circle $\partial\mathbb{D}$. Let $\gamma_n=1/|f'(z_n)|$. Then $(F_n)$ is of limit-disc type if and only if 
\[
\sum_{n=1}^\infty \gamma_1\dotsb \gamma_n < +\infty.
\]
Furthermore, if $(F_n)$ is of limit-disc type and converges ideally to a point $q$, then $q\neq z_0$.
\end{lemma}
\begin{proof}
 Let 
\[
\phi_n(z)=\frac{z+z_n}{-z+z_n},
\]
for $n=0,1,2,\dotsc$. One can check that $\phi_n(\mathbb{D})=\mathbb{K}$, $\phi_n(z_n)=\infty$ and $\phi_n(j)=j$, where, as usual,  $j=(0,0,1)$. Let $h_n=\phi_{n-1}f_n\phi_n^{-1}$ and $H_n=h_1\dotsb h_n$, so that $H_n=\phi_0F_n\phi_n^{-1}$. Observe that $h_n$ maps $\mathbb{K}$ strictly inside itself and $h_n(\infty)=\infty$. Also, $H_n(\mathbb{K})=\phi_0F_n(\mathbb{D})$, so $(F_n)$ is of limit-disc type if and only if $(H_n)$ is of limit-disc type. Now, the maps $\phi_n$ are isometries of $\overline{\mathbb{C}}$ with the chordal metric, so using the chain rule for chordal derivatives we obtain 
\[
h_n^\#(\infty)=(\phi_{n-1}f_n\phi_n^{-1})^\#(\infty)=f_n^\#(\phi_n^{-1}(\infty))=f_n^\#(z_n)=|f_n'(z_n)|.
\]
Also, writing $h_n(z)=a_nz+b_n$, we have $h_n^\#(\infty)=1/a_n$. We can now apply the first part of Lemma~\ref{pbe} to see that $(H_n)$ (and $(F_n)$) are of limit-disc type if and only if $\sum \gamma_1\dotsb \gamma_n < +\infty$.

To prove the second assertion, suppose that $(F_n)$ is of limit-disc type and converges ideally to a point $q$. Since $H_n(j)=\phi_0F_n(j)$, we see that $(H_n)$ converges ideally to $\phi_0(q)$. The second part of Lemma~\ref{pbe} tells us that $\phi_0(q)\neq\infty$. Hence $q\neq \phi_0^{-1}(\infty)=z_0$, as required.
\end{proof}

We can now prove Theorem~\ref{acu}.

\begin{proof}[Proof of Theorem~\ref{acu}]
Statement (i) of Theorem~\ref{acu} is equivalent to the assertion that $(F_n)$ is of limit-tangent type. Therefore Theorem~\ref{dgi} tells us that if $(F_n)$ is of limit-disc type, then statement~(i) holds. Thus, we have only to show that, on the assumption that statement~(i) holds, $(F_n)$ is of limit-disc type if and only if statement~(ii) holds.  

Let $z_n=\alpha_{f_n}$, for $n=1,2,\dotsc$, and $z_0=\beta_{f_1}$. There is a positive integer $m$ such that
\[
f_n(z_n)=f_n(\alpha_{f_n})=\beta_{f_n}=\alpha_{f_{n-1}}=z_{n-1},
\]
for $n>m$. We can now apply the first assertion of Lemma~\ref{abc} (ignoring the first $m$ terms of $(f_n)$) to deduce that $(F_n)$ is of limit-disc type if and only if $\sum \gamma_{f_1}\dotsb \gamma_{f_n}<+\infty$, as required.
\end{proof}

\section{Hausdorff dimension of sequences of limit-disc type}

In this section we prove Theorems~\ref{can} and \ref{cun}.

For each positive integer $b$, we define $\Omega_b$ to be the set of all sequences $x_1,x_2,\dotsc$ with $x_i\in\{0,\dots,b-1\}$. Given sequences $(x_n)$ and $(y_n)$ in $\Omega_b$, we define 
\[
d\big((x_n),(y_n)\big)= \frac{1}{b^k},\quad\text{where}\quad k= \min\{ i\in\mathbb{N}\,:\, x_i\neq y_i\}.
\]
Then $(\Omega_b,d)$ is a compact metric space; it is essentially the same as the metric space $(\Omega(\mathcal{F}),d)$ defined in the introduction, but with a different underlying set. We denote by $\text{dim}\, X$ the Hausdorff dimension of a subset $X$ of $\Omega_b$ with respect to the metric $d$.

Let $\Gamma$ be a directed graph with vertices $\{0,\dots,b-1\}$. There is at most one directed edge between each pair of vertices $i$ and $j$, where possibly $i=j$. We denote the spectral radius of the adjacency matrix of $\Gamma$ by $\rho(\Gamma)$. Let $\Gamma^\infty$ denote the set of infinite paths in $\Gamma$; specifically, $\Gamma^\infty$ consists of sequences $(x_n)$ from $\Omega_b$ such that there is a directed edge in $\Gamma$ from $x_n$ to $x_{n+1}$, for each  $n$. The following result is well known (see, for example, \cite[Corollary~2.9]{Fr1998}).

\begin{theorem}\label{coa}
Suppose that $\Gamma$ is a directed graph with vertices $\{0,1,\dots,b-1\}$, and suppose that $\Gamma$ contains a cycle. Then 
\[
\textnormal{dim}\, \Gamma^\infty = \frac{\log \rho(\Gamma)}{\log b}.
\]
\par\vspace{-2.2\baselineskip}
\qed
\end{theorem}

Theorem~\ref{can} follows quickly from this result.

\begin{proof}[Proof of Theorem~\ref{can}]
If $T(\mathcal{F})$ does \emph{not} contain a cycle, then there are no infinite paths in $T(\mathcal{F})$, so there are no composition sequences of limit-disc type. Consequently, $\text{dim}\,\Lambda(\mathcal{F})=0$. Suppose instead that $T(\mathcal{F})$ does contain a cycle.

Let $F_n=f_1\dotsb f_n$ be any composition sequence generated by $\mathcal{F}$. We are given that $\gamma_f<1$ for each $f\in \mathcal{F}$, from which it follows that $\sum \gamma_{f_1}\dotsb \gamma_{f_n} <+\infty$ (because $\mathcal{F}$ is finite). We can now see from Theorem~\ref{acu} that 
\[
\Lambda(\mathcal{F}) = \bigcup_{k=1}^\infty Y_k,
\]
where $Y_k$ comprises those sequences $(f_n)$ in $\mathcal{F}$ for which there is a directed edge in $T(\mathcal{F})$ from $f_n$ to $f_{n+1}$ for $n\geq k$. Observe that $Y_1$ is the set of all infinite paths in $T(\mathcal{F})$.

Define $b=|\mathcal{F}|$, and choose any one-to-one map from $\mathcal{F}$ to $\{0,\dots,b-1\}$. This map induces an isometry of the metric spaces $(\Omega(\mathcal{F}),d)$ and $(\Omega_b,d)$. We can now apply Theorem~\ref{coa} to give $\text{dim}\, Y_1 = (\log \rho(\mathcal{F})) /\log b$. Furthermore, each set $Y_1\setminus Y_k$ is finite, so $\text{dim}\, Y_k=\text{dim}\, Y_1$. Finally, because Hausdorff dimension is countably stable, we obtain
\[
\text{dim}\, \Lambda(\mathcal{F}) = \text{dim}\,\bigcup_{k=1}^\infty Y_k = \sup_{k\in\mathbb{N}} \text{dim}\, Y_k =  \frac{\log \rho(\mathcal{F})}{\log b}.\qedhere
\]
\end{proof}

The remainder of this section is concerned with proving Theorem~\ref{cun}. For most of this part we work with the unit interval $[0,1]$ and the Euclidean metric rather than with the metric space $(\Omega(\mathcal{F}),d)$. Only at the very end do we switch back to the latter metric space. We denote the Hausdorff dimension of a subset $X$ of $[0,1]$ by $\text{dim}\, X$, as usual. Throughout, we let $b$ be a fixed positive integer (which later will be the order of $\mathcal{F}$).

We make use of two well-known lemmas (see, for example, \cite{Fa1997}*{Proposition~2.3}  and \cite{Bi1965}*{equation~14.4}). The first lemma is part of a result known as Billingsley's lemma. In this lemma we use the notation $I_n(x)$ to denote the $n$th generation, half-open $b$-adic interval of the form $\left[\frac{j-1}{b^n},\frac{j}{b^n}\right)$ containing~$x$, and we let $|I_n(x)|$ denote the width $1/b^n$ of this interval.

\begin{lemma}\label{bil}
Suppose that $\mu$ is a finite Borel measure on $[0,1]$ and $A$ is a Borel subset of $[0,1]$ with $\mu(A)>0$. If 
\[
 \liminf_{n \rightarrow \infty} \frac{\log \mu(I_{n}(x))}{\log | I_{n}(x)|} \leq \lambda, 
\]
 for all $x\in A$, then $\textnormal{dim}\, A\leq \lambda$.\qed
\end{lemma}

Given $x\in [0,1]$, we define $x(j)$ to be the $j$th digit in the $b$-ary expansion of $x$. (To ensure that this definition is unambiguous, we forbid infinitely-recurring entries $b-1$ in the expansion.) For $i=0,\dots,b-1$, we let $N_i(x,n)$ denote the number of those digits $x(1),\dots,x(n)$ that are equal to $i$. We also need the concept of a  \emph{probability vector}, which is a vector $(p_0, \ldots , p_{b-1})$ such that $0 \leq p_{i} \leq 1$ for each $i$ and $p_0+\dots +p_{b-1}=1$. 

\begin{lemma}\label{bol}
Suppose that $p=(p_0, \ldots , p_{b-1})$ is a probability vector. Let $X_p$ be the set of those numbers $x$ in $[0,1]$ such that 
\[
\lim_{n\to\infty} \frac{N_i(x,n)}{n} = p_i,
\]
for $i=0,\dots,b-1$. Then
\[ 
\textnormal{dim}\, X_p = -\frac{1}{\log b}\sum_{i=0}^{b-1} p_{i} \log p_{i}. 
\]
\par\vspace{-2.6\baselineskip}
\qed
\end{lemma}

Given positive numbers $\gamma_0,\dots,\gamma_{b-1}$, we define, for $x\in[0,1]$,
\[ 
Q_n(x)= \frac{1}{n}\sum_{i=0}^{b-1} N_{i}(x,n)\log \gamma_{i}.
 \]
 
\begin{lemma}\label{vby}
If $\limsup_{n \rightarrow \infty} Q_n(x)<0$, then the series 
\[
\sum_{n=1}^{\infty} \gamma_{x(1)}\dotsb\gamma_{x(n)} 
\] 
converges, and if $\limsup_{n \rightarrow \infty} Q_n(x)>0$, then the series diverges.
\end{lemma}
\begin{proof}
Applying Cauchy's root test, we see that the series converges if $\limsup {a_n}^{1/n}<1$ and diverges if  $\limsup {a_n}^{1/n}>1$, where $a_n=\gamma_{x(1)}\dotsb\gamma_{x(n)} $. The results follows by taking the logarithm of each side of each of these inequalities, as
\[
\log\left(\limsup_{n\to\infty} {a_n}^{1/n} \right)=\limsup_{n\to\infty} \frac{1}{n}\sum_{j=1}^n\log \gamma_{x(j)} = \limsup_{n\to\infty} Q_n(x).\qedhere
\]
\end{proof}

We can now state the principal result of this section.

\begin{theorem}\label{dit}
Suppose that at least one of the positive numbers $\gamma_0,\dots,\gamma_{b-1}$ is greater than $1$, and at least one is less than $1$.  Let $X$ denote the set of numbers $x$ in $[0,1]$ for which
\[
 \sum_{n=1}^{\infty} \gamma_{x(1)}\dotsb\gamma_{x(n)} <+\infty.
\]
Then
\[
\textnormal{dim}\, X=  \min_{s \geq 0}  \frac{\log \left(\sum_{i=0}^{b-1} \gamma_{i}^{-s}\right)}{\log b}. 
\]
\end{theorem}

Before proving this theorem, we observe that the expression for $\text{dim}\,X$ can be written more concisely using the real function $q(s) =  \sum_{i=0}^{b-1} \gamma_{i}^{-s}$. This function satisfies the equation $q{''}(s)=\sum_{i=0}^{b-1} (\log\gamma_i)^2\gamma_{i}^{-s}>0$, so $q$ is strictly convex. Since $q(s)\to +\infty$ as $s\to -\infty$ and as $s\to +\infty$, we see that $q$ has a unique local minimum at a point $s_0$, and $s_0$ is a global minimum of $q$. Now let $g(s)=\log q(s)$. Then $\text{dim}\, X = g(s_0)/\log b$ if $s_0\geq 0$, and $\text{dim}\, X = g(0)/\log b =1$ if $s_0<0$.

In the argument that follows, we also use the function $g'(s)$, which is strictly increasing on the real line and satisfies $g'(s_0)=0$.

\begin{proof}[Proof of Theorem~\ref{dit}]
We begin by establishing a lower bound for $\text{dim}\, X$. Let $X_s$ denote the set of those numbers $x$ in $[0,1]$ such that
\[
\lim_{n\to\infty} \frac{N_i(x,n)}{n} = \frac{\gamma_i^{-s}}{\sum_j \gamma_j^{-s}},
\]
for $i=0,\dots,b-1$. If $x\in X_s$, then 
\[
\lim_{n\to\infty} Q_n(x) = \frac{\sum_i \gamma_i^{-s}\log\gamma_i}{\sum_j \gamma_j^{-s}}=-g'(s).
\]
If $s>s_0$, then $g'(s)>0$, so $X_s\subset X$, by Lemma~\ref{vby}. Therefore, applying Lemma~\ref{bol}, we obtain
\[
\text{dim}\, X \geq \text{dim}\, X_s = \frac{g(s)}{\log b}-\frac{sg'(s)}{\log b}.
\]
 If $s_0<0$, then $\text{dim}\, X \geq g(0)/\log b =1$, so $\text{dim}\,X=1$. Henceforth we assume that $s_0\geq 0$. In this case, as $g'(s_0)=0$, we obtain $\text{dim}\, X \geq g(s_0)/\log b$.
 
It remains to prove that $\text{dim}\, X \leq g(s_0)/\log b$ when $s_0\geq 0$. To this end, we define $\mu_s$ to be the probability measure given by 
\[
\mu_s(I_n(x)) =p_{x(1)}\dotsb p_{x(n)},\quad\text{where}\quad (p_0,\dots,p_{b-1}) = \frac{1}{\sum_i \gamma_i^{-s}}(\gamma_0^{-s},\dots,\gamma_{b-1}^{-s})
\]
is the probability vector used already, and $I_n(x)$ is the $n$th generation $b$-adic interval containing~$x$. 

Let $Y$ be the set of those numbers $x$ in $[0,1]$ such that $\limsup_{n\to\infty}Q_n(x)\leq 0$. Then $X\subset Y$, by Lemma~\ref{vby}. One can check that, given $x\in Y$,
\begin{align*}
\liminf_{ n \rightarrow \infty} \frac{\log \mu_{s}(I_{n}(x))}{\log |I_{n}(x)|} 
&= \frac{g(s)}{\log b} +  \frac{s}{\log b} \liminf_{n \rightarrow \infty} Q_n(x)\\
&\leq \frac{g(s)}{\log b} +  \frac{s}{\log b} \limsup_{n \rightarrow \infty} Q_n(x)\\
&\leq \frac{g(s)}{\log b},
\end{align*}
for all $s\geq s_0$. Lemma~\ref{bil} now tells us that 
\[
\text{dim}\, X \leq \text{dim}\, Y \leq \frac{g(s)}{\log b}. 
\]
Therefore $\text{dim}\, X \leq g(s_0)/\log b$, as required.
\end{proof}

Finally, we prove Theorem~\ref{cun}.

\begin{proof}[Proof of Theorem~\ref{cun}]
If $\gamma_f\geq 1$ for all $f\in\mathcal{F}$, then condition (ii) of Theorem~\ref{acu} is not satisfied for any composition sequence generated by $\mathcal{F}$, so $\Lambda(\mathcal{F})=\emptyset$. Hence $\text{dim}\,\Lambda(\mathcal{F})=0$. On the other hand, if $\gamma_f<1$ for all $f\in\mathcal{F}$, then conditions~(i) and~(ii) of Theorem~\ref{acu} are satisfied for every composition sequence generated by $\mathcal{F}$, so $\text{dim}\, \Lambda(\mathcal{F})=1$.

Suppose now that at least one of the numbers $\gamma_f$ is greater than $1$, and at least one is less than $1$. Let $b=|\mathcal{F}|$. As before, we choose any one-to-one map $\phi$ from $\mathcal{F}$ to $\{0,\dots,b-1\}$.  Let $\gamma_0,\dots,\gamma_{b-1}$ be the positive numbers defined by $\gamma_{\phi(f)}=\gamma_f$, for $f\in\mathcal{F}$. Now define $\pi$ to be the projection from $\Omega(\mathcal{F})$ to $[0,1]$ that takes the sequence $(f_n)$ to the number $x$ with $n$th term $\phi(f_n)$ in its $b$-ary expansion. The map $\pi$ is one-to-one outside a countable subset of $\Omega(\mathcal{F})$, and one can prove using an elementary argument  about Hausdorff dimension (see \cite{Fa1986}*{Theorem~5.1}) that for any subset $A$ of $\Omega(\mathcal{F})$ we have $\text{dim}\, A= \text{dim}\, \pi(A)$.

Theorem~\ref{acu} tells us that $\Lambda(\mathcal{F})$  comprises those sequences $(f_n)$ for which $\sum_{n}\gamma_{f_1}\dotsb \gamma_{f_n}<+\infty$.  Let $\Lambda'(\mathcal{F})$ denote the class $\Lambda(\mathcal{F})$ minus any sequences $(f_n)$ for which $f_n=\phi^{-1}(b-1)$ for all sufficiently large values of $n$. The classes $\Lambda(\mathcal{F})$ and $\Lambda'(\mathcal{F})$ differ by a countable set, so they have the same Hausdorff dimension. The set $\pi(\Lambda'(\mathcal{F}))$ comprises those numbers $x$ in $[0,1]$ for which $\sum_n \gamma_{x(1)}\dotsb \gamma_{x(n)}<+\infty$. We can now apply Theorem~\ref{dit} to deduce that
\[
\text{dim}\, \Lambda(\mathcal{F})=\text{dim}\, \pi(\Lambda(\mathcal{F}))=  \min_{s \geq 0}  \frac{\log \left(\sum_{i=0}^{b-1} \gamma_{i}^{-s}\right)}{\log b}=\min_{s \geq 0}  \frac{\log \left(\sum_{f\in\mathcal{F}} \gamma_{f}^{-s}\right)}{\log b}.\qedhere
\]
\end{proof}

\section{Convergence of composition sequences}

In this section we prove Theorems~\ref{ibi} and~\ref{mwd}. Our proofs use the action of M\"obius transformations on $\mathbb{H}^3$, with hyperbolic metric $\rho$. Given a point $z=(x,y,t)$ in $\mathbb{H}^3$, we define $h(z)=t$: the `height' of $z$. Also, for a Euclidean disc $D$ in the complex plane with centre $c$ and radius $r$, we denote by $\Pi(D)$ the hyperbolic plane $\{z\in \mathbb{H}^3\,:\,|z-c|=r\}$. For brevity, we write $\Pi$ for the hyperbolic plane $\Pi(\mathbb{D})$.

\begin{lemma}\label{cul}
Let $U$ and $V$ be distinct Euclidean discs with centres $u$ and $v$, and radii $r$ and $s$, in that order. Suppose that $V\subset U$ and 
\begin{equation}\label{ksg}
\frac{r}{s} < \min\left\{2,1+\tfrac18\sinh \rho(z,\Pi(V))\right\},
\end{equation}
for some point $z$ in $\Pi(U)$.  Then $h(z)<s/2$. Let $w$ be the point in $\Pi(V)$ with $h(w)=h(z)$ that is closest in Euclidean distance to $z$. Then
\[
|z-w|<3(r-s).
\]
\end{lemma}
\begin{proof}
Using a standard formula for the hyperbolic metric (see, for example, \cite[Section 7.20]{Be1995}) we have
\[
\sinh \rho(z,\Pi(V)) =\frac{|z-v|^2-s^2}{2hs},
\]
where $h=h(z)$. As $V\subset U$, we see that $|u-v|+s\leq r$. Hence
\[
|z-v|^2-s^2 \leq (|z-u|+|u-v|)^2-s^2\leq (2r-s)^2-s^2=4r(r-s).
\] 
Therefore, using \eqref{ksg}, we obtain
\[
\sinh \rho(z,\Pi(V)) \leq \frac{2r(r-s)}{hs}=\frac{2r}{h}\left(\frac{r}{s}-1\right)<\frac{r}{4h}\sinh \rho(z,\Pi(V)).
\]
It follows that $h<r/4<s/2$. 

Now, let $\Sigma$ denote the Euclidean plane in $\mathbb{H}^3$ with height $h$, that is, $\Sigma=\{(x,y,h)\,:\, x,y\in\mathbb{R}\}$.  Let $U_0$ and $V_0$ denote the Euclidean discs in $\Sigma$ with centres $u+hj$ and $v+hj$ (where $j=(0,0,1)$), and radii $\sqrt{r^2-h^2}$ and $\sqrt{s^2-h^2}$, in that order. Observe that $V_0\subset U_0$ and $z\in \partial U_0$ and $w\in \partial V_0$. Using elementary Euclidean geometry, we can see that 
\[
|z-w|< 2\sqrt{r^2-h^2}-2\sqrt{s^2-h^2}.
\]
But $h<s/2$ and $h<r/2$, so
\[
2\sqrt{r^2-h^2}-2\sqrt{s^2-h^2}=\frac{2(r^2-s^2)}{\sqrt{r^2-h^2}+\sqrt{s^2-h^2}}<\frac{4}{\sqrt3}(r-s).
\]
Therefore $|z-w|<3(r-s)$, as required.
\end{proof}

We now prove Theorem~\ref{ibi}. 

\begin{proof}[Proof of Theorem~\ref{ibi}]
Since $\text{rad}(f_n(\mathbb{D}))<\delta$ for all $n$, it follows that each hyperbolic plane $f_n(\Pi)$ lies inside the region $t<\delta$ in $\mathbb{H}^3$. Hence $\rho(j,f_n(\Pi))> \rho(j,\delta j)= -\log\delta$. Let $\varepsilon=-\log\delta$, and define $\eta=\min\left\{2,1+\tfrac18\sinh\varepsilon\right\}$. Let $r_n$ be the Euclidean radius of $F_n(\mathbb{D})$ (with $r_0=1$).  Define
\[
A=\left\{n\in\mathbb{N}\,:\, \frac{r_{n-1}}{r_n}<\eta\right\}\quad\text{and}\quad B=\left\{n\in\mathbb{N}\,:\, \frac{r_{n-1}}{r_n}\geq \eta\right\}.
\]
Suppose that $n\in A$. Observe that $\rho(F_{n-1}(j),F_n(\Pi))= \rho(j,f_n(\Pi))>\varepsilon$, so
\[
\frac{r_{n-1}}{r_n}<\eta < \min\left\{2,1+\tfrac18\sinh \rho(F_{n-1}(j),F_n(\Pi))\right\}.
\]
By Lemma~\ref{cul}, we can choose a point $w_n$ in $F_n(\Pi)$ with $h(w_n)=h(F_{n-1}(j))$ that is closest in Euclidean distance to $F_{n-1}(j)$, and
\[
|F_{n-1}(j)-w_n| <3(r_{n-1}-r_n).
\]
Now, using a basic estimate of hyperbolic distance, we find that
\[
\varepsilon h(F_{n-1}(j)) < \rho(F_{n-1}(j),F_n(\Pi))h(F_{n-1}(j)) \leq \rho(F_{n-1}(j),w_n)h(F_{n-1}(j))\leq |F_{n-1}(j)-w_n|. 
\]
Hence $\varepsilon h(F_{n-1}(j)) < 3(r_{n-1}-r_n)$, so
\[
\sum_{n\in A} h(F_{n-1}(j))<+\infty.
\]

Next, let $n_1<n_2<\dotsb$ be the elements of $B$. Define $n_0=0$.  As $r_{n_k}/r_{n_{k-1}}\leq 1/\eta$, for $k=1,2,\dotsc$,  we see that $r_{n_k}\leq 1/\eta^k$. Therefore
\[
h(F_{n_k-1}(j)) \leq r_{n_k-1}\leq r_{n_{k-1}} \leq \frac{1}{\eta^{k-1}}, 
\]
so
\[
\sum_{n\in B} h(F_{n-1}(j))<+\infty.
\]

To finish, observe that $-\log h(F_{n-1}(j))\leq \rho(j,F_{n-1}(j))$, from which it follows that 
\[
\exp(-\rho(j,F_{n-1}(j)))\leq h(F_{n-1}(j)).
\]
 We have just seen that 
\[
\sum_{n=1}^\infty h(F_{n-1}(j))<+\infty, 
\]
and hence we conclude that $(F_n)$ is a rapid-escape sequence.
\end{proof}

Finally, we prove Theorem~\ref{mwd}.

\begin{proof}[Proof of Theorem~\ref{mwd}]
Let us begin by proving that the sequence $(F_n)$ converges ideally. In this part of the proof (and nowhere else in the paper) it is convenient to assume that the transformations $F_n$ act on the unit ball model of hyperbolic space, which we denote by $\mathbb{B}^3$. The ideal boundary of $\mathbb{B}^3$ is the unit sphere. We denote the hyperbolic metric on $\mathbb{B}^3$ by $\rho$. In this model we use the distinguished point $0$ (the origin) instead of the point $j$ in $\mathbb{H}^3$. Two standard formulas for the hyperbolic metric on $\mathbb{B}^3$ are 
\[
\rho(0,z) = \log \left(\frac{1+|z|}{1-|z|}\right)\quad\text{and}\quad \sinh \tfrac{1}{2} \rho(z,w) = \frac{|z-w|}{\sqrt{(1-|z|^2)(1-|w|^2)}}.
\]
From the first standard formula we see that  $1-|z|=(1+|z|)e^{-\rho(0,z)}$. Substituting this equation into the second standard formula, and noting that $1+|z|\leq 2$,  we obtain
\[
|z-w| = \sqrt{(1-|z|^2)(1-|w|^2)}\sinh \tfrac{1}{2} \rho(z,w)\leq 2(e^{-\rho(0,z)}+e^{-\rho(0,w)})\sinh \tfrac{1}{2}\rho(z,w).
\]
We apply this formula with $z=F_{n-1}(0)$ and $w=F_n(0)$. Let $k>0$ be such that $\rho(0,f_n(0))<k$ for all $n$ (the maps $f_n$ are chosen from a finite set). Then $\rho(F_{n-1}(0),F_n(0))<k$ for all $n$. Therefore
\[
|F_{n-1}(0)-F_n(0)| < 2(e^{-\rho(0,F_{n-1}(0))}+e^{-\rho(0,F_{n}(0))})\sinh \tfrac12 k.
\]
Theorem~\ref{ibi} tells us that $(F_n)$ is a rapid-escape sequence, so  $\sum |F_{n-1}(0)-F_n(0)| <+\infty$. It follows that $(F_n(0))$ converges in the Euclidean metric on $\overline{\mathbb{B}^3}$, the closed unit ball. This sequence cannot converge to a point in $\mathbb{B}^3$, because $(F_n)$ is an escaping sequence. Therefore $(F_n)$ converges ideally to a point on the ideal boundary, the unit sphere.

Let us now go back to thinking about $F_n$ acting on $\overline{\mathbb{C}}$ and $\mathbb{H}^3$. We have seen that $(F_n)$ converges ideally to a point $q$. This point $q$ must belong to $\overline{\mathbb{D}}$ because each transformation $f_n$ maps the closed hyperbolic half-space with ideal boundary $\overline{\mathbb{D}}$ inside itself.

Next we show that $(F_n)$ converges locally uniformly to $q$ on $\overline{\mathbb{D}}\setminus X$. To prove this, we begin with the observation (see \cite[Theorem~4.6]{Ae1990}) that if $(G_n)$ is a sequence of M\"obius transformations that converges ideally to a point $u$, then $(G_n)$ converges locally uniformly to $u$ on the complement in $\overline{\mathbb{C}}$ of its backward limit set (the backward limit set of $(G_n)$ is the set of accumulation points, in the chordal metric, of the sequence $(G_n^{-1}(j))$). In our case, each transformation $f_n^{-1}$ maps the complement of the unit disc, namely $\overline{\mathbb{C}}\setminus \mathbb{D}$, inside itself, so it also maps the closed hyperbolic half-space in $\mathbb{H}^3$ with ideal boundary $\overline{\mathbb{C}}\setminus \mathbb{D}$ (which includes the point $j$) inside itself. It follows that the backward limit set of $(F_n)$ is contained in $\overline{\mathbb{C}}\setminus \mathbb{D}$. Therefore $(F_n)$ converges locally uniformly to $q$ on $\mathbb{D}$.

Now choose a point $y$ in $\partial\mathbb{D}\setminus X$. Let 
\[
\mathcal{G} = \{f\in\mathcal{F}\,:\,\text{$f_n=f$ for infinitely many integers $n$}\}.
\]
Let $K$ be a compact disc in $\mathbb{D}$ whose interior contains the finite set $\{f(y)\,:\, f\in\mathcal{G}\}$. Let $E$ be a closed Euclidean disc in the complement of $X$ that contains the point $y$ in its interior and is such that $f(E)\subset K$ for all $f\in\mathcal{G}$. We can choose $m$ to be a sufficiently large positive integer that if $n\geq m$ then $f_n\in\mathcal{G}$. Therefore $f_n(E)\subset K$, so the sequence $(F_n)$ converges uniformly to $q	$ on $E$. Hence $(F_n)$ converges locally uniformly to $q$ on $\overline{\mathbb{D}}\setminus X$.

Let us finish by considering convergence of $(F_n)$ on the set $X$ itself. Choose $x\in X$. If $(F_n)$ is of limit-point type, then $(F_n)$ converges uniformly to $q$ on $\overline{\mathbb{D}}$, so $F_n(x)\to q$ as $n\to\infty$. Suppose now that $(F_n)$ is of limit-disc type. We must show that $(F_n(x))$ does not converge to $q$. By omitting the first $m$ terms from the sequence $(f_n)$, for some positive integer $m$, and by adjusting $q$ accordingly, we can assume that $f_n(z_n)=z_{n-1}$, for $n=1,2,\dotsc$, where $z_n=\alpha_{f_n}$, for $n\geq 1$, and $z_0$ is a point common to the boundaries of all the discs $F_n(\mathbb{D})$. Lemma~\ref{abc} tells us that $z_0\neq q$. However, we know that $z_n=x$ for infinitely many positive integers $n$, and $F_n(z_n)=z_0$; therefore $(F_n(x))$ does not converge to $q$. 
\end{proof}

\end{document}